\documentclass[a4paper, 10pt]{amsart}
\usepackage{amsthm}
\usepackage[]{amsmath}
\usepackage{amssymb}
\usepackage{enumerate}
\usepackage{tabularx}
\usepackage[]{color}
\usepackage[left=2.2cm, right=2.2cm, top=2.5cm, bottom=2.5cm]{geometry}
\usepackage[colorlinks]{hyperref}
\usepackage{tikz}
\usepackage{multirow}
\usepackage{subcaption}
\usepackage{algorithm}
\usepackage{algorithmic}
\usepackage{xcolor}

\newcommand{\bm}[1]{\boldsymbol{#1}}
\newcommand{\bmr}[1]{\bm{\mr{#1}}}
\newcommand{\lj}{[ \hspace{-2pt} [}
\newcommand{\rj}{] \hspace{-2pt} ]}
\newcommand{\mb}[1]{\mathbb{#1}}
\newcommand{\mc}[1]{\mathcal{#1}}
\newcommand{\mr}[1]{\mathrm{#1}}
\newcommand{\jump}[1]{\lj #1 \rj}
\newcommand{\aver}[1]{ \{#1\}  }
\newcommand{\wt}[1]{ \widetilde{ #1}}

\newcommand{\DGnorm}[1]{ \| #1\|_{\mr{DG}}}
\newcommand{\DGenorm}[1]{ |\!|\!| #1 |\!|\!|}

\renewcommand{\d}[1]{\mathrm d \boldsymbol{#1}}

\def\curl{\ifmmode \mathrm{curl} \else \text{curl}\fi}
\def\Curl{\ifmmode \mathrm{Curl} \else \text{Curl}\fi}
\def\div{\ifmmode \mathrm{div} \else \text{div}\fi}
\def\Div{\ifmmode \mathrm{Div} \else \text{Div}\fi}
\def\dim{\ifmmode \mathrm{dim} \else \text{dim}\fi}
\def\MTh{\mc{T}_h}
\def\MEh{\mc{E}_h}

\def\un{\bm{\mr n}}
\def\Ned{\ifmmode \text{N\'ed\'elec} \else \text{N\'ed\'elec} \fi}

\newcommand\comment[1]{}

\newcommand\substitute[2]{#2}
\newcommand\rrrevise[1]{#1}

\newtheorem{assumption}{Assumption}
\newtheorem{theorem}{Theorem}
\newtheorem{lemma}{Lemma}
\newtheorem{corollary}{Corollary}

\newtheorem{remark}{Remark}

\definecolor{orange}{rgb}{1, 0.5, 0}

\allowdisplaybreaks[4]

\title[Fourth-order Curl Problem]{An arbitrary order Reconstructed
Discontinuous Approximation to Fourth-order Curl Problem}

\author[R. Li]{Ruo Li} \address{CAPT, LMAM and School of Mathematical
Sciences, Peking University, Beijing 100871, P.R. China}
\email{rli@math.pku.edu.cn}

\author[Q.-C. Liu]{Qicheng Liu} \address{School of Mathematical
Sciences, Peking University, Beijing 100871, P.R. China}
\email{qcliu@pku.edu.cn}

\author[S.-H. Zhao]{Shuhai Zhao} \address{School of Mathematical
Sciences, Peking University, Beijing 100871, P.R. China}
\email{shuhai@pku.edu.cn}

\begin{document}
\maketitle

\begin{abstract}
We present an arbitrary order discontinuous Galerkin finite element
method for solving the fourth-order curl problem using a reconstructed
discontinuous approximation method. It is based on an arbitrarily high-order 
approximation space with one unknown
per element in each dimension. The discrete 
problem is based on the symmetric IPDG method. We prove
\emph{a priori} error estimates under the energy norm and the $L^2$ 
norm and show numerical results to verify the theoretical analysis.
\\
 \textbf{keywords}: fourth-order curl problem, 
patch reconstruction, discontinuous Galerkin method
\end{abstract}

\section{Introduction}
\label{sec_introduction}
We are concerned in this paper with the fourth-order curl problem,
which has applications in inverse electromagnetic
scattering, and magnetohydrodynamics (MHD) when modeling magnetized 
plasmas.
Discretizing the fourth-order curl operator is one of the keys to 
simulate these models. 
Additionally, the fourth-order curl operator plays an important role in 
approximating the Maxwell transmission eigenvalue problem. 
Therefore, it is important to design highly efficient and accurate 
numerical methods for fourth-order curl problems 
\cite{Cakoni2007AVA,monk2012}.

Finite element methods (FEMs) are a widely used numerical scheme for
solving partial differential equations.
The design of FEMs for fourth-order curl problems is challenging. Nevertheless, 
in recent years, researchers have proposed and analysed various 
FEMs for these problems to address this difficulty. According to 
whether the finite element space is contained within the space of the 
exact solution, finite element methods can be classified into 
conforming methods and nonconforming methods. For conforming methods, 
one needs to construct $H(\curl^2)$-conforming finite elements, we
refer to \cite{curlconform2019,curlconform2020} for some recent works.
Due to the high order of the fourth-order curl operator, it is still a
difficult task to implement a high-order conforming space. Therefore,
many researches focus on nonconforming methods. The discontinuous
Galerkin (DG) methods use completely discontinuous piecewise
polynomial to approximate the solution of partial differential
equations. These methods enforce numerical solutions to be close to
the exact solution by adding penalties. DG methods are
commonly used in solving complex problems due to their simplicity and
flexibility in approximation space. For fourth-order curl
problems, we refer to \cite{zheng2010quadcurl,JCM565, wg2019quadcurl,
mix2013quadcurl,mix2018quadcurl} for some DG methods.

A significant drawback of DG methods is the large number of degrees of 
freedom in DG space, which results in high computational costs. This
drawback is a matter of concern. 
In this paper, we propose an arbitrary order discontinuous Galerkin 
FEM for solving the fourth-order curl problem.
The method is based on a reconstructed approximation space
that has only one unknown per dimension in each element. The
construction of the approximation space includes
creating an element patch per element and solving a local least 
squares fitting problem to obtain a local high-order polynomial. 
Methods based on the reconstructed spaces are called reconstructed 
discontinuous approximation (RDA) methods, which have been 
successfully applied to a series of classical problems
\cite{Li2020interface, Li2016discontinuous, Li2019reconstructed,
Li2019sequential,Li2023curl, Li2019least}. 
The reconstructed space is a subspace of the standard DG space,
which can approximate functions to high-order accuracy and
inherits the flexibility on the mesh partition in the meanwhile. One advantage of this
space is that it has very few degrees of freedom, which gives high 
approximation efficiency of finite element. 
Using this new space, we design the discrete scheme for the 
fourth-order curl problem under the symmetric interior penalty 
discontinuous Galerkin (IPDG) framework.
We prove the convergence rates under
the energy norm and the $L^2$ norm, and numerical experiments
are conducted to verify the theoretical analysis ,showing our 
algorithm is simple to implement and can reach high-order accuracy. 

The rest of this paper is organized as follows. In Section
\ref{sec_preliminaries}, we introduce the fourth-order curl problem without div-free condition
and give the basic notations about the Sobolev spaces and the
partition. We also recall two commonly used inequalities in this
section. In Section \ref{sec_space}, we establish the reconstruction
operator and the corresponding approximation space. Some basic
properties of the reconstruction are also proven. In
Section \ref{sec_curl_problem}, we define the discrete
variational form for the fourth-order curl problem and analyse the error
under the energy norm and the $L^2$ norm, and prove that the convergence 
rate is optimal with respect to the energy norm. In Section
\ref{sec_numericalresults}, we carried out some numerical examples to
validate our theoretical results. Finally, a brief conclusion is given in Section
\ref{sec_conclusion}.


\section{Preliminaries}
\label{sec_preliminaries}
Let $\Omega \subset \mb{R}^d (d = 2, 3)$ be a bounded polygonal
(polyhedral) domain with a Lipschitz boundary $\partial \Omega$. 
\\
Given a bounded domain $D \subset \Omega$, we follow the standard
definitions to the space $L^2(D), L^2(D)^d$, the spaces $H^q(D),
H^q(D)^d$ with the regular exponent $q \geq 0$. The corresponding
norms and seminorms are defined as
\begin{displaymath}
  \| \cdot \|_{H^q(D)} := \left(  \| \cdot
  \|_{H^q(D)}^2 \right)^{1/2}, \,\, | \cdot |_{H^q(D)} 
  := \left(  | \cdot |_{H^q(D)}^2 \right)^{1/2},
\end{displaymath}
respectively. 

Throughout this paper, for two vectors $\bm{a} = (a_1, \ldots, a_d) \in
\mb{R}^d, \bm{b} = (b_1, \ldots, b_d) \in \mb{R}^d$, the cross product
$\bm{a} \times \bm{b}$ is defined as 
\begin{displaymath}
  \bm{a} \times \bm{b} := \begin{cases}
    a_1b_2 - a_2b_1, \quad d = 2, \\
    (a_2b_3 - a_3b_2, a_3b_1 - a_1b_3, a_1b_2 - a_2b_1)^T, \quad d =
    3. \\
  \end{cases}
\end{displaymath}
The cross product between a vector and a scalar will be
used in two dimensions, and in this case for a vector $\bm{a} = (a_1,
a_2) \in \mb{R}^2$ and a scalar $b \in \mb{R}$,  $\bm{a} \times b$ is
defined as $\bm{a} \times b := (a_2 b, -a_1b)^T$.
Specifically, for a vector-valued function $\bm{u} \in \mb{R}^d$, 
the curl of $\bm{u}$ reads
\begin{displaymath}
  \curl\bm{u} :=  \nabla \times \bm{u} = \begin{cases}
    \frac{\partial u_2 }{\partial x} - \frac{\partial u_1}{\partial
    y}, \quad d = 2, \\
    (\frac{\partial u_3}{\partial y} - \frac{\partial u_2}{\partial
    z}, \frac{\partial u_1}{\partial z} - \frac{\partial u_3}{\partial
    x}, \frac{\partial u_2}{\partial x} - \frac{\partial u_1}{\partial
    y} )^T, \quad d = 3,
  \end{cases}
\end{displaymath}
and for the scalar-valued function $q(x,y)$, we let $\nabla \times q$ 
be the formal adjoint, which reads
\begin{displaymath}
  \nabla \times q = \left( \frac{\partial q}{\partial y},
  - \frac{\partial q}{\partial x} \right)^T.
\end{displaymath}
For convenience, we mainly use the notations for the three-dimensional
case. 

In this paper, we consider the following fourth-order curl problem
\begin{equation}
 \left\{
  \begin{aligned}
    &\curl^{4}\bm{u}+\bm{u} = \bm{f}, &\quad \text{ in } \Omega,
    \\
    & \bm{u}\times\bm{n} = g_1,  &\quad \text{ on } \partial \Omega,
    \\
    &  (\nabla \times \bm{u})\times\bm{n} = g_2, &\quad \text{ on } \partial \Omega,
  \end{aligned}
 \right.
 \label{eq_quadcurl}
\end{equation}
For the problem domain $\Omega$, we define the space 
\begin{displaymath}
  \begin{aligned}
    H(\curl^s, \Omega) &:= 
      \left\{ \bm{v} \in L^2(\Omega)^d \ | \ \curl^j \bm{v} \in
      L^2(\Omega), 1\le j\le s \right\},  
  \end{aligned}
\end{displaymath}
and
\begin{displaymath}
  \begin{aligned}
    H_0(\curl^2, \Omega) &:= 
      \left\{ \bm{v} \in L^2(\Omega)^d \ | \ \curl^j \bm{v} \in
      L^2(\Omega), \ \curl^{j-1} \bm{v} = 0 \ \text{on}\  \partial \Omega,1\le j\le 2 \right\},  
  \end{aligned}
\end{displaymath}
For the weak formulation we are going to find 
$\bm{u}\in H(\curl^2,\Omega)$ such that 
 \begin{displaymath}
  \begin{aligned}
    a(\bm{u},\bm{v})=(\bm{f},\bm{v})_{L^2(\Omega)},\ \forall \bm{v}\in
    H_0(\curl^2,\Omega),
  \end{aligned}
\end{displaymath}
where 
\begin{displaymath}
  a(\bm{u}, \bm{v}) := (\curl^2 \bm{u}, \curl^2 \bm{v})_{L^2(\Omega)}
  + (\bm{u}, \bm{v})_{L^2(\Omega)}.
\end{displaymath}
We refer to \cite{regcurl, mix2018quadcurl} for some regularity results of this problem.

Next, we define some notations about the mesh. Let $\MTh$ be a regular 
and quasi-uniform partition $\Omega$ into disjoint open triangles 
(tetrahedra). Let $\MEh$ denote the set of all $d-1$
dimensional faces of $\MTh$, and we decompose $\MEh$ into $\MEh = 
\MEh^i \cup \MEh^b$, where $\MEh^i$ and $\MEh^b$ are the 
sets of interior faces and boundary faces, respectively. We let 
\begin{displaymath}
  h_K := \text{diam}(K), \quad \forall K \in \MTh, \quad h_e :=
  \text{diam}(e), \quad \forall e \in \MEh, 
\end{displaymath}
and define $h := \max_{K \in \MTh} h_K$.
The quasi-uniformity of the mesh $\MTh$ is in the sense that there
exists a constant $\nu > 0$ such that $h \leq \nu \min_{K \in
\MTh} \rho_K$, where $\rho_K$ is the diameter of the
largest ball inscribed in $K$. One can get the inverse inequality and
the trace inequality from the regularity of the mesh, which are
commonly used in the analysis.
\begin{lemma}
  There exists a constant $C$ independent of the mesh size $h$, such that
  \begin{equation}
    \| v \|^2_{L^2(\partial K)} \leq C \left( h_K^{-1} \| v 
    \|^2_{L^2(K)} + h_K \| \nabla v \|_{L^2(K)}^2 \right), \quad
    \forall v \in H^1(K).
    \label{eq_trace}
  \end{equation}
  \label{le_trace}
\end{lemma}
\begin{lemma}
There exists a constant $C$ independent of the mesh size $h$, such that
  \begin{equation}
    \| v \|_{H^q(K)} \leq C h_K^{-q} \| v \|_{L^2(K)}, \quad 
    \forall v \in \mb{P}_l(K),
    \label{eq_inverse}
  \end{equation}
  where $\mb{P}_l(K)$ is the space of polynomial on $K$ with the 
  degree no more than $l$.
  \label{le_inverse}
\end{lemma}
We refer to \cite{Brenner2007mathematical} for more details of these 
inequalities.

Finally, we introduce the trace operators that will be used in our 
numerical schemes. For $e \in \MEh^i$, we denote by $K^+$ and $K^-$ the
two neighbouring elements that share the boundary $e$, and $\un^+, \un^-$
the unit out normal vector on $e$, respectively. We define the jump
operator $\jump{\cdot}$ and the average operator $\aver{\cdot}$ as
\begin{equation}
  \begin{aligned}
    \jump{\bm{q}} &:=  \bm{q}|_{K^+}\times\un^+ + \bm{q}|_{K^-}\times \un^- 
    , \quad \text{ for vector-valued function}, \\
    \jump{\curl\bm{q}} &:= \curl \bm{q}|_{K^+}\times\un^ + + \curl\bm{q}|_{K^-}\times \un^- 
    , \quad \text{ for vector-valued function}, \\
    \aver{v} &:= \frac{1}{2}(v|_{K^+}  + v|_{K^-}), \quad \text{ for
    scalar-valued function}, \\
    \aver{\bm{q}} &:= \frac{1}{2}(\bm{q}|_{K^+} + \bm{q}|_{K^-}), 
    \quad \text{ for vector-valued function}.
  \end{aligned}
  \label{eq_traceopei}
\end{equation}
For $e \in \MEh^b$, we let $K \in \MTh$ such that $e \in \partial K$ 
and $\un$ is the unit out normal vector. We define 
\begin{equation}
  \begin{aligned}
    & \jump{\bm{q}} := \bm{q}|_{K}\times \un , \quad \jump{\curl\bm{q}} := \curl \bm{q}|_{K}\times \un
    \text{ for vector-valued function}, \\
    &\aver{v} := v|_{K}\quad \text{ for scalar-valued function}, \qquad
    \aver{\bm{q}} := \bm{q}|_{K} \quad \text{ for vector-valued
    function}.
  \end{aligned}
  \label{eq_traceopeb}
\end{equation}

Throughout this paper, $C$ and $C$ with subscripts denote the generic 
constants that may differ between lines but are independent of the 
mesh size.

\section{Reconstructed Discontinuous Space}
\label{sec_space}
In this section, we will introduce a linear reconstruction operator to
obtain a discontinuous approximation space for the given mesh $\MTh$.
The reconstructed space can achieve a high-order accuracy while the
number of degrees of freedom in each dimension remain the same as the number of
elements in $\MTh$.  The construction of the operator includes the
following steps.

Step 1. For each $K \in \MTh$, we construct an element patch $S(K)$,
which consists of $K$ itself and some surrounding elements. 
The size of the patch is controlled with a given threshold $\# S$. 
The construction of the element patch $S(K)$ is conducted by a
recursive algorithm. We begin by setting $S_0(K) = \{ K \}$, and
define $S_t(K)$ recursively: 
\begin{equation}
  S_t(K) = \bigcup_{K' \in S_{t-1}(K)} \bigcup_{ K'' \in \Delta(K')}
  K'', \quad t = 0,1, \dots
  \label{eq_patch}
\end{equation}
where $\Delta(K):= \{ K' \in \MTh \ | \  \overline{K} \cap \overline{K'}
\neq \emptyset \}$. The recursion stops once $t$ meets the condition 
that the cardinality $\# S_t(K) \geq \# S$, and we let the patch 
$S(K) := S_t(K)$. We apply the recursive algorithm \eqref{eq_patch} 
to all elements in $\MTh$.

Step 2. For each $K \in \MTh$, we solve a local least squares fitting
problem on the patch $S(K)$. We let $\bm{x}_K$ be the barycenter of
the element $K$ and mark barycenters of all elements as
collocation points.  
Let $I(K)$ be the set of collocation points located inside the domain
of $S(K)$,
\begin{displaymath}
  I(K) := \{ \bm{x}_{K'} \ |\ K' \in S(K) \},
\end{displaymath}
Let $U_h^0$ be the piecewise constant space with respect to $\MTh$,
\begin{displaymath}
  U_h^0 := \{ v_h \in L^2(\Omega) \ | \  v_h|_K \in \mb{P}_0(K), \  
  \forall K \in \MTh\}.
\end{displaymath}

and we denote by $\bmr{U}_h^0 := (U_h^0)^d$ the vector-valued
piecewise constant space. Given a function $\bm{g}_h \in \bmr{U}_h^0$
and an integer $m \geq 1$, for the element $K \in \MTh$,
we consider the following local constrained least squares problem: 
\begin{equation}
  \bm{p}_{S(K)} = \mathop{\arg \min}_{ \bm{q} \in
  \mb{P}_m(S(K))^d} \sum_{\bm{x} \in I(K)} \|\bm{q}(\bm{x}) -
  \bm{g}_h(\bm{x}) \|_{l^2}^2, \quad \text{s.t. } 
  \bm{q}(\bm{x}_K) = \bm{g}_h(\bm{x}_K), 
  \label{eq_lsproblem}
\end{equation}
where $I(K) := \{ \bm{x}_{\wt{K}} \ | \  \wt{K} \in S(K) \cap
\MTh\}$ denotes the set of collocation points located in $S(K)$,
and $\| \cdot \|_{l^2}$ denotes the discrete $l^2$ norm for vectors.
We make the following
geometrical assumption on the location of collocation points
\cite{Li2012efficient, Li2016discontinuous}: 
\begin{assumption}
  For any element patch $S(K)$ and any polynomial $p \in
  \mb{P}_m(S(K))$, $p|_{I(K) } = 0$ implies $p|_{S(K)} = 0$.
  \label{as_lsproblem}
\end{assumption}
\rrrevise{
The Assumption \ref{as_lsproblem} excludes the case that the points in
$I(K)$ are on an algebraic curve of degree $m$ and requires
cardinality $\# I(K) \geq \dim(\mb{P}_m(\cdot))$.} Under this
assumption, the fitting problem \eqref{eq_lsproblem} admits a unique
solution.
By solving \eqref{eq_lsproblem} and considering
the restriction $(\bm{p}_{S(K)})|_K$, we actually gain a local
polynomial defined on $K$. Since $\bm{p}_{S(K)}$ is sought in the
least squares sense, we note that $\bm{p}_{S(K)}$ depends linearly
on the given function $\bm{g}_h$. This property inspires us to define
a linear operator from the piecewise constant space $\bmr{U}_h^0$ to a
piecewise polynomial space with respect to $\MTh$.

We define a local reconstruction
operator $\mc{R}_K$ for elements in $\MTh$, 
\begin{displaymath}
  \begin{aligned}
    \mc{R}_K : \bmr{U}_h^0  &\rightarrow (\mb{P}_m(K))^d, \\
    \bm{g}_h & \rightarrow (\bm{p}_{S(K)})|_K, \\
  \end{aligned} \quad \forall K \in \MTh.
\end{displaymath}
Further, the linear reconstruction operator $\mc{R}$ for $\MTh$
is defined in a piecewise manner as 
\begin{displaymath}
  \begin{aligned}
    \mc{R} : \bmr{U}_h^0 & \rightarrow \bmr{U}_h^{m}, \\
    \bm{g}_h & \rightarrow \mc{R} \bm{g}_h, \\
  \end{aligned} \quad
  (\mc{R} \bm{g}_h)|_K := \mc{R}_K \bm{g}_h,   
  \quad \forall K \in \MTh, 
\end{displaymath}
where $\bmr{U}_h^{m}$ is the image space of the operator
$\mc{R}$. Clearly, by the operator $\mc{R}$, any
piecewise constant function $\bm{g}_h$ will be mapped into a piecewise
polynomial function $\mc{R} \bm{g}_h \in \bmr{U}_h^{m}$. Then, we
present more details to the space $\bmr{U}_h^{m}$, which are
piecewise polynomial spaces for the partition $\MTh$. 
For any element $K$, we pick up a function $\bm{w}_{K, j} \in
\bmr{U}_h^0$ such that 
\begin{displaymath}
  \bm{w}_{K, j}(\bm{x}) = \begin{cases}
    \bm{e}_j, & \bm{x} \in K, \\
    \bm{0}, & \text{otherwise}, \\
  \end{cases} \quad 1 \leq j \leq d, 
\end{displaymath}
where $\bm{e}_j$ is the unit vector in $\mb{R}^d$ whose $j$-th entry
is $1$. We let $\bm{\lambda}_{K, j} := \mc{R} \bm{w}_{K, j} (K \in
\MTh \leq j \leq d)$ and we state that the space
$\bmr{U}_h^{m}$ is spanned by $ \{ \bm{\lambda}_{K, j} \}$. 
\begin{lemma}
  the functions $ \{ \bm{\lambda}_{K, j} \}(K \in
  \MTh, 1 \leq j \leq d)$ are linearly independent
  and the space $\bmr{U}_h^{m} = \text{span}( \{ \bm{\lambda}_{K,
  j} \})$. 
  \label{le_lambdaid}
\end{lemma}
\begin{proof}
  For any $K \in \MTh$, assume there exists a group of coefficients
  $\{a_{K, j}\}$ such that  
  \begin{equation}
    \sum_{K \in \MTh} \sum_{j = 1}^d a_{K, j}
    \bm{\lambda}_{K, j}(\bm{x}) = 0, \quad \forall \bm{x} \in
    \mb{R}^d.
    \label{eq_alambda}
  \end{equation}
  From the constraint to the problem \eqref{eq_lsproblem}, we have
  that 
  \begin{displaymath}
    \bm{\lambda}_{K, j}(\bm{x}_{K'}) = \begin{cases}
      \bm{e}_j, & K' = K, \\
      \bm{0}, & \text{otherwise}. \\
    \end{cases}
  \end{displaymath}
  We choose $\bm{x} = \bm{x}_K$ for all $K \in \MTh$ in
  \eqref{eq_alambda}, and it can be seen that $a_{K, j} = 0$. Thus,
  the functions $ \{ \bm{\lambda}_{K, j} \}(K \in \MTh, 1 \leq j
  \leq d)$ are linearly independent. Obviously, there holds
  $\bm{\lambda}_{K, j} \in \bmr{U}_h^{m}$, and we note that the
  dimension of $ \{ \bm{\lambda}_{K, j} \}$ is $d \# \MTh$. By the
  linear property of $\mc{R}$, we have that the space $\bmr{U}_h^{m}$
   is spanned by $ \{ \bm{\lambda}_{K, j} \}$. This completes the
  proof.
\end{proof}
Lemma~\ref{le_lambdaid} claims that the basis functions of
$\bmr{U}_h^{m}$ are indeed the group of functions $\{
\bm{\lambda}_{K, j} \}$, and for any function $\bm{g}_h \in
\bmr{U}_h^0$, one can explicitly write $\mc{R} \bm{g}_h$ as 
\begin{equation}
  \substitute{\mc{R} \bm{g}_h = \sum_{K \in \MTh} \sum_{j = 1}^d
  \bm{g_h}(\bm{x}_K) \bm{\lambda}_{K, j}.}
  {\mc{R} \bm{g}_h = \sum_{K \in \MTh} \sum_{j = 1}^d
  g_{h,j}(\bm{x}_K) \bm{\lambda}_{K, j}.}
  \label{eq_Rig}
\end{equation}
From the problem \eqref{eq_lsproblem}, the basis function
$\bm{\lambda}_{K, j}$ vanishes on the element $K'$ that $K \not\in
S(K')$.  This fact indicates $\bm{\lambda}_{K, j}$ has a finite
support set that $\text{supp}(\bm{\lambda}_{K, j}) = \{K' \in \MTh
\ | \ K \in S(K') \}$, and $\text{supp}(\bm{\lambda}_{K, j})$ is
different from the patch $S(K)$.

We extend the operator $\mc{R}$ to act on smooth
functions. For any $\bm{g} \in H^{m+1}(\Omega)^d$, we define a
piecewise constant function $\bm{g}_h \in \bmr{U}_h^0$ as 
\begin{displaymath}
  \bm{g}_h(\bm{x}_K) := \bm{g}(\bm{x}_K), \quad \forall K \in \MTh,
\end{displaymath}
and we directly define $\mc{R} \bm{g} := \mc{R} \bm{g}_h$. In this
way, any smooth function in $H^{m+1}(\Omega)^d$ is mapped into a
piecewise polynomial function with respect to $\MTh$ by the operator
$\mc{R}$. For any $\bm{g} \in H^{m+1}(\Omega)^d$, $\mc{R} \bm{g}$
can also be written as \eqref{eq_Rig}.

Next, we will focus on the approximation property of the operator
$\mc{R}$. We first define a constant $\Lambda_{m,K}$
for every element patch,
\begin{displaymath} 
  \Lambda^2_{m,K} := \max_{p \in
  \mb{P}_m(S(K))} \frac{\|p\|_{L^2}
  }{h^d\sum_{\bm{x}\in I(K)}p(\bm{x})^2 }.
\end{displaymath}
Assumption \ref{as_lsproblem} as well as the norm equivalence in finite dimensional spaces
 actually ensures $\Lambda_{m,K} <
\infty$.
\begin{lemma}
  For any element $K \in \MTh$, there holds
  \begin{equation}
    \| \mc{R}_K \bm{g} \|_{L^2(K)} \leq (1+2
    \Lambda_{m,K} \sqrt{\# I(K)})h^{d/2} \max_{\bm{x}\in I(K)} |\bm{g}|
    , \quad \forall \bm{g} \in H^{m+1}(\Omega)^d. 
    \label{eq_stability}
  \end{equation}
  \label{le_stability}
\end{lemma}
\begin{proof}
  We define a polynomial space 
  \begin{displaymath}
    \wt{\mb{P}}_m(S(K))^d := \left\{ \bm{v} \in \mb{P}_m(S(K))^d
    \ | \  \bm{v}(\bm{x}_{K}) = \bm{g}(\bm{x}_{K}) \right\}.
  \end{displaymath}
  Clearly, any polynomial in $\wt{\mb{P}}_m(S(K))^d$ satisfies the
  constraint to \eqref{eq_lsproblem}. Since $\bm{p} :=
  \bm{p}_{S(K)}$ is the solution to \eqref{eq_lsproblem}, for any
  \substitute{$\varepsilon > 0$}{$\varepsilon \in \mb{R}$} and 
  any $\bm{q} \in \wt{\mb{P}}_m(S(K))^d$, we
  have that $\bm{p} + \varepsilon ( \bm{q} - \bm{g}(\bm{x}_{K})) 
  \in \wt{\mb{P}}_m(S(K))^d$ and 
  \begin{displaymath}
    \sum_{\bm{x} \in I(K)} \| \bm{p}(\bm{x}) + \varepsilon (
    \bm{q}(\bm{x}) - \bm{g}(\bm{x}_{K}))- \bm{g}(\bm{x})
    \|_{l^2}^{2} \geq \sum_{\bm{x} \in I(K)} \| \bm{p}(\bm{x}) -
    \bm{g}(\bm{x}) \|_{l^2}^{2}.
  \end{displaymath}
  Because $\varepsilon$ is arbitrary, the above inequality implies 
  \begin{displaymath}
    \sum_{\bm{x} \in I(K)} (\bm{p}(\bm{x}) - \bm{g}(\bm{x})) \cdot
    (\bm{q} (\bm{x}) - \bm{g}(\bm{x}_{K})) = 0.
  \end{displaymath}
  for any $\bm{q} \in \wt{\mb{P}}_m(S(K))^d$. By letting $\bm{q} =
  \bm{p}$ and applying Cauchy-Schwarz inequality, \substitute{we
  obtain that}{}
  \rrrevise{
    \begin{displaymath}
    \begin{aligned}
      0 &= \sum_{\bm{x} \in I(K)} (\bm{g}(\bm{x}) -
      \bm{g}(\bm{x}_{K}) + \bm{g}(\bm{x}_{K}) - \bm{p}(\bm{x}))
      \cdot (\bm{p}(\bm{x}) - \bm{g}(\bm{x}_{K}))
      \\
      & = \sum_{\bm{x} \in I(K)} \left( -\| \bm{p}(\bm{x}) -
      \bm{g}(\bm{x}_{K})\|^2_{l^2} + (\bm{p}(\bm{x}) -
      \bm{g}(\bm{x}_{K}))
      \cdot (\bm{g}(\bm{x}) - \bm{g}(\bm{x}_{K})) \right) \\
      & \leq -\frac{1}{2} \sum_{\bm{x} \in I(K)} \| \bm{p}(\bm{x}) -
      \bm{g}(\bm{x}_{K}) \|^2_{l^2} + \frac{1}{2} \sum_{\bm{x} \in I(K)}
      \| \bm{g}(\bm{x}) - \bm{g}(\bm{x}_{K}) \|^2_{l^2},
    \end{aligned}
  \end{displaymath}
  }
  \rrrevise{and we obtain}
  \begin{equation}
    \sum_{\bm{x} \in I(K)} \| \bm{p}(\bm{x}) -
    \bm{g}(\bm{x}_{K}) \|_{l^2}^2 \leq  \sum_{\bm{x} \in I(K)} 
    \| \bm{g}(\bm{x}) - \bm{g}(\bm{x}_{K}) \|_{l^2}^2.
    \label{eq_pgh1}
  \end{equation}
  Combining the definition of the constant$ \Lambda_{m,K}$, we get 
 
  \begin{displaymath}
    \begin{aligned}
      \| \bm{p}-\bm{g}(\bm{x}_K) \|_{L^2(K)} &\leq \Lambda^2_{m,K} h^d 
      \sum_{\bm{x} \in I(K)} \| \bm{p}(\bm{x}) -
    \bm{g}(\bm{x}_{K}) \|_{l^2}^2 \leq \Lambda^2_{m,K} h^d  \sum_{\bm{x} \in I(K)} 
    \| \bm{g}(\bm{x}) - \bm{g}(\bm{x}_{K}) \|_{l^2}^2\\
    &\le 4d \Lambda^2_{m,K} h^d \# I(K) \max_{\bm{x}\in I(K)} |\bm{g}|^2,
    \end{aligned}
  \end{displaymath}
  hence
  \begin{displaymath}
    \begin{aligned}
      \| \bm{p} \|_{L^2(K)} &\leq \| \bm{p}-\bm{g}(\bm{x}_K) \|_{L^2(K)}+h^{d/2}|\bm{g}(\bm{x}_K)|
    &\le (1+2
    \Lambda_{m,K} \sqrt{\# I(K)})h^{d/2} \max_{\bm{x}\in I(K)} |\bm{g}|
    \end{aligned}
  \end{displaymath}
  and completes the proof.
\end{proof}
\begin{assumption}
  For every element patch $S(K)(K \in \MTh)$,
  there exist constants $R$ and $r$ which are independent of $K$ such
  that $B_r \subset S(K) \subset B_R$, and $S(K)$ is star-shaped
  with respect to $B_r$, where $B_\rho$ is a disk with the radius
  $\rho$. 
  \label{as_patch}
\end{assumption}
From the stability result \eqref{eq_stability}, we can prove the
approximation results.  
\begin{lemma}
  For any $K \in \MTh $, there exists a constant $C$ such
  that 
  \begin{equation}
    \begin{aligned}
      \| \bm{g} - \mc{R} \bm{g} \|_{H^q(K)} &\leq C \Lambda_m h_K^{m
      + 1 - q} \| \bm{g} \|_{H^{m+1}(S(K))}, \quad 0 \leq q \leq m\\
    \end{aligned}
    \label{eq_localapproximation}
  \end{equation}
  for any $\bm{g} \in H^{m+1}(\Omega)^d$, where we set 
  \begin{equation}
    \Lambda_m :=  \max_{K \in \MTh} \left( 1 +
    \Lambda_{m,K} \sqrt{\# I(K) } \right).
    \label{eq_Lambdam}
  \end{equation}
  \label{eq_const}
\end{lemma}
\begin{proof}
  According to Assumption\ref{as_patch}, there exists a polynomial $\bm{p}\in\mb{P}_m(S(K))^d$, such that
  \begin{displaymath}
    \begin{aligned}
      \| \bm{g} - \bm{p} \|_{H^q(S(K))} &\leq h_K^{m+ 1 - q} \| \bm{g} \|_{H^{m+1}(S(K))}\\
      \| \bm{g} - \bm{p} \|_{L^\infty(S(K))} &\leq h_K^{m+ 1 - d/2} \| \bm{g} \|_{H^{m+1}(S(K))}\\
    \end{aligned}
  \end{displaymath}
  Combining the stablility result of the reconstruction operator, we have
  \begin{displaymath}
    \begin{aligned}
      \| \bm{g} - \mc{R}\bm{g} \|_{L^2(K)} &\leq \| \bm{g} - \bm{p} \|_{L^2(K)} +\| \bm{p} - \mc{R}\bm{g} \|_{L^2(K)} \\
   &\leq  \| \bm{g} - \bm{p} \|_{L^2(K)}+(1+2
   \Lambda_{m,K} \sqrt{\# I(K)})h^{d/2} \max_{\bm{x}\in I(K)}|\bm{p}(\bm{x})-\bm{g}(\bm{x})|
    \end{aligned}
  \end{displaymath}
  By the arbitraryness of $\bm{p}$,
  \begin{displaymath}
    \begin{aligned}
      \| \bm{g} - \mc{R}\bm{g} \|_{L^2(K)} &\leq Ch^{m+1}\Lambda_m\| \bm{g} \|_{H^{m+1}(S(K))}
    \end{aligned}
  \end{displaymath}
  the case $q>0$ follows from the inverse inequality
  \begin{displaymath}
    \begin{aligned}
      \| \bm{g} - \mc{R}\bm{g} \|_{H^q(K)} &\leq \| \bm{g} - \bm{p} \|_{H^q(K)} +\| \bm{p} - \mc{R}\bm{g} \|_{H^q(K)} \\
   &\leq  Ch^{m+1-q}\| \bm{g} \|_{H^{m+1}(K)}+ Ch^{-q}\| \bm{p} - \mc{R}\bm{g} \|_{L^2(K)}\\
   &\leq  C\Lambda_m h^{m+1-q}\| \bm{g} \|_{H^{m+1}(S(K))}
    \end{aligned}
  \end{displaymath}
\end{proof}
From \eqref{eq_localapproximation}, it can be seen that the operator
$\mc{R}$ has an approximation error of degree $O(h^{m+1-q})$ under
the case that $\Lambda_m$ admits \substitute{a}{an} upper bound 
independent of the mesh size $h$. However, it is not trivial to bound 
the constant $\Lambda_m$, see Remark \ref{re_Lambda}. We can 
prove that under some mild conditions on the element patch, the 
constant admits a uniform upper bound. 

\begin{remark}
  There is a constant $\Lambda(m,S(K))$ defined in \cite{Li2016discontinuous} as
  \begin{displaymath} 
    \Lambda(m, S(K)) := \max_{p \in
    \mb{P}_m(S(K))} \frac{\max_{\bm{x} \in S(K)} |p(\bm{x})|
    }{\max_{\bm{x} \in I(K)} |p(\bm{x})| }.
  \end{displaymath}
  It can be shown that $\Lambda_{m,K} \le \Lambda(m,S(K))$. Hence when the constant $\Lambda(m,S(K))$
  is uniformly bounded, so is $\Lambda_{m,K}$.
  For the patch $S(K)$, consider the special case that the
  corresponding collocation points set $ I(K)$ has that $\# I(K) =
  \dim(\mb{P}_m(\cdot))$, then the constant $\Lambda(m, S(K))$ is
  equal to the Lebesgue constant \cite{Powell1981approximation} and
  the solution to the least squares problem \eqref{eq_lsproblem} is
  just the Lagrange interpolation polynomial, however, in this case the constant
  $\Lambda(m, S(K))$ may depend on $h$ and grow very fast, which will
  further damage the convergence. To ensure the stability of the reconstruction
  operator, we are acquired to
  choose a sufficiently large $\# S$ so that the constant
  $\Lambda(m, S(K))$ admits a uniform upper bound, thus. We refer to 
  \cite[Lemma 6]{Li2016discontinuous} and 
  \cite[Lemma 3.4]{Li2012efficient} for the details of this statement.
  The wide element patch will bring more computational cost
  for filling the stiffness matrix and increase the width of the
  banded structure. This is just the price we have to pay for the
  uniform bound of $\Lambda(m, S(K))$.
  
  \label{re_Lambda}
\end{remark}

\section{Approximation to Fourth-order Curl Problem}
\label{sec_curl_problem}
We define the approximation problem to solve the
problem \eqref{eq_quadcurl} based on the space $\bmr{U}_h^m$ constructed in
the previous section: Seek $\bm{u_h} \in \bmr{U}_h^m$ ($m\geq 2$) such that
\begin{equation}
  B_h(\bm{u_h}, \bm{v_h}) = l_h(\bm{v_h}), \quad \forall \bm{v_h} \in \bmr{U}_h^m,
  \label{eq_variational}
\end{equation}
where we define the bilinear form $B_h(\cdot, \cdot)$ on
$ U_h := \bmr{U}_h^m + H(\curl^2,\Omega)\cap\{\bm{v}:\| \curl^j\bm{v}\|_{1/2+\sigma,\Omega}<\infty, 0\le j\le 3\}$,
as in \cite{zhimin2023}.
\begin{displaymath}
  \begin{aligned}
    B_h(\bm{u}, \bm{v}) &:= \sum_{K \in \MTh} \int_{K} \curl ^{2}\bm{u} \cdot \curl ^{2}\bm{v}\
     \d{\bm{x}} + \sum_{K \in \MTh} \int_{K} \bm{u} \cdot \bm{v}\
     \d{\bm{x}} 
     \\
     &+ \sum_{e \in \MEh} \int_{e} 
    \left( \jump{\bm{u}} \cdot \aver{\curl^{3} \bm{v}} + 
    \jump{\curl \bm{u}} \cdot \aver{\curl^{2}\bm{v}} \right) \d{\bm{s}}
    + \sum_{e \in \MEh} \int_{e} 
    \left( \jump{\bm{v}} \cdot \aver{\curl^{3} \bm{u}} + 
    \jump{\curl \bm{v}} \cdot \aver{\curl^{2}\bm{u}} \right) \d{\bm{s}}
    \\
    &+ \sum_{e \in \MEh} \int_{e} \left(
    \mu_1 \jump{ \bm{u}} \cdot \jump{\bm{v}} + \mu_2 \jump{\curl \bm{u}}\cdot
    \jump{\curl \bm{v}} \right) \d{\bm{s}} 
  \end{aligned}
\end{displaymath}
The parameter $\mu_1$ and $\mu_2$ are positive penalties which are set
by
\begin{displaymath}
  \begin{aligned}
    \mu_1|_e = \frac{\eta}{h_e^3}, \quad \mu_2|_e =
    \frac{\eta}{h_e}. \quad \text{ for } e \in \MEh.
  \end{aligned}
\end{displaymath}
The linear form $l_h(\cdot)$ is defined for $\bm{v} \in U_h$,
\begin{displaymath}
  \begin{aligned}
    l_h(\cdot) &:= \sum_{K \in \MTh} \int_K \bm{f}\cdot\bm{v} \d{\bm{x}} 
    &+ \sum_{e \in \MEh^b} \int_{e} \left( g_1 \cdot(\curl^{3}\bm{v_h}
     + \mu_1 \bm{v_h}\times \un) + g_2 \cdot(\curl^{2}\bm{v_h} + \mu_2 \curl\bm{v_h}\times\un) \right)
    \d{\bm{s}}.\\
  \end{aligned}
\end{displaymath}
We introduce two energy norms for the space $U_h$,
\begin{displaymath}
  \begin{aligned}
    \DGnorm{\bm{u}}^2 &:= \sum_{K \in \MTh} \int_{K} | \bm{u}
    |^2 \d{\bm{x}} +\sum_{K \in \MTh} \int_{K} | \curl^{2} \bm{u}
  |^2 \d{\bm{x}} + 
   \sum_{e \in \MEh} \int_{e} h_K^{-3} |
  \jump{\bm{u}} |^2 \d{\bm{s}} 
  + \sum_{e \in \MEh} \int_{e} h_K^{-1} | \jump{\curl \bm{u}} |^2
  \d{\bm{s}},
  \end{aligned}
\end{displaymath}
and
\begin{displaymath}
  \begin{aligned}
    \DGenorm{\bm{u}}^2 &:= \DGnorm{\bm{u}}^2 + \sum_{e \in \MEh} \int_{e} h_e^3 | \aver{\curl^{3} \bm{u}} |^2 \d{\bm{s}} 
    + \sum_{e \in \MEh} \int_{e} h_e | \aver{\curl^{2} \bm{u}} |^2
    \d{\bm{s}} .
  \end{aligned}
\end{displaymath}

We claim that the two energy norms are equivalent over the space 
$\bmr{U}_h^m$.
\begin{lemma}
  For any $\bm{u_h} \in \bmr{U}_h^m$, there exists a constant $C$, such that
  \begin{equation}
    \DGnorm{\bm{u_h}} \leq \DGenorm{\bm{u_h}} \leq C \DGnorm{\bm{u_h}}.
    \label{eq_normEq}
  \end{equation}
  \label{le_normEq}
\end{lemma}
\begin{proof}
  We only need to prove $\DGenorm{\bm{u_h}} \leq C \DGnorm{\bm{u_h}}$. For $e
  \in \MEh^i$, we denote the two neighbor elements of $e$ by $K^+$ and
  $K^-$. We have 
  \begin{displaymath}
    \| h_e^{3/2} \aver{\curl^{3} \bm{u_h}} \|_{L^2(e)}
    \leq C  \left(\| h_e^{3/2} \curl^{3} \bm{u_h}
    \|_{L^2(e \cap \partial K^+)} + \| h_e^{3/2} \curl^{3} \bm{u_h}
    \|_{L^2(e \cap \partial K^-)} \right) 
  \end{displaymath}
  By the trace inequalities \eqref{eq_trace},
  and the inverse inequality \eqref{eq_inverse}, we obtain that
  \begin{displaymath}
    \|h_e^{3/2} \curl^{3} \bm{u_h}\|_{L^2(e \cap \partial K^{\pm})}
    \leq 
      C \| \curl^{2} \bm{u_h} \|_{L^2(K^{\pm})}
  \end{displaymath}
  For $e \in \MEh^b$, let $e \subset K$. Similarly, we have
  \begin{displaymath}
    \|h_e^{3/2} \curl^{3} \bm{u_h}\|_{L^2(e \cap \partial K)}
    \leq 
    C \| \curl^{2} \bm{u_h} \|_{L^2(K)}
  \end{displaymath}
  
  The term $\| h_e^{1/2} \aver{\curl^{2} \bm{u_h}} \|_{L^2(e)}$ can bounded by the 
  same way. Thus, by summing over all $e \in \MEh$ ,
  we conclude that
  \begin{displaymath}
    \begin{aligned}
    \sum_{e \in \MEh} \| h_e^{1/2} \aver{\curl^{2} \bm{u_h}} \|_{L^2(e)}^2 
    + \sum_{e \in \MEh} \| h_e^{3/2} \aver{\curl^{3} \bm{u_h}}
    \|_{L^2(e)}^2  \leq C \sum_{K \in \MTh} \| \curl^{2} \bm{u_h} \|_{L^2(K)}^2,
    \end{aligned}
  \end{displaymath}
  which can immediately leads us to \eqref{eq_normEq} and completes
  the proof.
\end{proof}
Now we are ready to prove the coercivity and the continuity of the
bilinear form $B_h(\cdot, \cdot)$.
\begin{theorem}
  Let $B_h(\cdot, \cdot)$ be the bilinear form with sufficiently large
  penalty $\eta$. Then there exists a positive constant $C$ such that 
  \begin{equation}
    B_h(\bm{u_h}, \bm{u_h}) \geq C \DGenorm{\bm{u_h}}^2, \quad \forall \bm{u_h} \in \bmr{U}_h^m.
    \label{eq_coercivity}
  \end{equation}
  \label{th_coercivity}
\end{theorem}
\begin{proof}
  From Lemma \ref{le_normEq}, we only need to establish the coercivity
  over the norm $\DGnorm{\cdot}$. For the face $e \in \MEh^i$, let
  $e$ be shared by the neighbor elements $K^-$ and $K^+$. We apply the
  Cauchy-Schwarz inequality,
  \begin{displaymath}
    \begin{aligned}
      -\int_{e} 2\jump{\bm{u_h}} \cdot &\aver{\curl^{3}
      \bm{u_h}} \d{\bm{s}} \geq -\frac{1}{\epsilon}\|h_e^{-3/2}
      \jump{\bm{u_h}}\|^2_{L^2(e)} - \epsilon \|h_e^{3/2} 
      \aver{\curl^{3} \bm{u_h}} \|^2_{L^2(e)} \\
      &\geq  -\frac{1}{\epsilon}\|h_e^{-3/2}
      \jump{\bm{u_h}} \|^2_{L^2(e)} -\frac{\epsilon}{2} \| h_e^{3/2} \curl^{3} \bm{u_h}
       \|^2_{L^2(e \cap \partial K^-)} - \frac{\epsilon}{2} \|
      h_e^{3/2} \curl^{3} \bm{u_h} \|^2_{L^2(e \cap \partial
      K^+)} ,
    \end{aligned}
  \end{displaymath}
  for any $\epsilon > 0$. From the trace inequalities
  \eqref{eq_trace}, and the inverse inequality
  \eqref{eq_inverse}, we deduce that
  \begin{displaymath}
    \|h_e^{3/2}  \curl^{3} \bm{u_h}\|_{L^2(e^i \cap \partial K^{\pm})}
    \leq 
      C \|  \curl^{2} \bm{u_h} \|_{L^2(K^{\pm})}
  \end{displaymath}
  Thus, we have
  \begin{equation}
    -\sum_{e \in \MEh^i} \int_{e^0 \cup e^1} 2 \jump{\bm{u_h}} \cdot
    \aver{ \curl^{3} \bm{u_h}} \d{\bm{s}} \geq -\sum_{e \in
    \MEh^i} \frac{1}{\epsilon} \| h_e^{-3/2} \jump{\bm{u_h}}
    \|^2_{L^2(e^0 \cup e^1)} - C\epsilon \sum_{K \in \MTh} \|  \curl^{2} \bm{u_h}
    \|^2_{L^2(K^0 \cup K^1)}.
    \label{eq_eo1}
  \end{equation}
  For the face $e \in \MEh^b$, we can similarly derive that
  \begin{equation}
    -\sum_{e \in \MEh^b} \int_{e^0 \cup e^1} 2 \jump{\bm{u_h}} \cdot
    \aver{\curl^{3}\bm{u_h}} \d{\bm{s}} \geq -\sum_{e \in
    \MEh^b} \frac{1}{\epsilon} \| h_e^{-3/2} \jump{\bm{u_h}}
    \|^2_{L^2(e)} - C\epsilon \sum_{K \in \MTh} \| \curl^{2} \bm{u_h}
    \|^2_{L^2(K)}.
    \label{eq_eb1}
  \end{equation}

  By employing the same method to the term $\int_{e^0 \cup e^1} 2 
  \jump{ \curl \bm{u_h}} \aver{ \curl^{2} \bm{u_h}} \d{\bm{s}}$ , we can obtain that
  \begin{equation}
    -\sum_{e \in \MEh} \int_{e} 2 \jump{\curl \bm{u_h}}
    \aver{\curl^{2} \bm{u_h}} \d{\bm{s}} \geq -\frac{1}{\epsilon} \sum_{e
    \in \MEh} \| h_e^{-1/2} \jump{\curl \bm{u_h}} \|^2_{L^2(e)}
    - C\epsilon \sum_{K \in \MTh} \| \curl^{2}\bm{u_h} \|^2_{L^2(K)},
    \label{eq_eob2}
  \end{equation}

  Combining the inequalities \eqref{eq_eo1}, \eqref{eq_eb1},
  \eqref{eq_eob2},  we conclude that
  there exists a constant $C$ such that
  \begin{displaymath}
    \begin{aligned}
      B_h(\bm{u_h}, \bm{u_h}) &\geq (1 - C\epsilon) 
      \sum_{K \in \MTh} \|  \curl^{2} \bm{u_h}\|^2_{L^2(K)} +\sum_{K \in \MTh} \|\bm{u_h}\|^2_{L^2(K)}\\
    &+ (\eta - \frac{1}{\epsilon}) \sum_{e
    \in \MEh} (\| h_e^{-1/2} \jump{ \curl \bm{u_h}} \|^2_{L^2(e)} 
    + \| h_e^{-3/2} \jump{\bm{u_h}} \|^2_{L^2(e)}) 
    \end{aligned}
  \end{displaymath}
  for any $\epsilon > 0$. We can let $\epsilon = 1/(2C)$ and
  select a sufficiently large $\eta$ to ensure $B_h(\bm{u_h}, \bm{u_h}) \geq C
  \DGnorm{\bm{u_h}}^2$, which completes the proof.
\end{proof}
\begin{theorem}
  There exists a positive constant $C$ such that
  \begin{equation}
    |B_h(\bm{u},\bm{v})| \leq C \DGenorm{\bm{u}} \DGenorm{\bm{v}}, \quad \forall \bm{u},\bm{v} \in
    U_h.
    \label{eq_continuity}
  \end{equation}
  \label{th_continuity}
\end{theorem}
\begin{proof}
  By directly using the Cauchy-Schwarz
  inequality,
  \begin{displaymath}
    \begin{aligned}
      B_h(\bm{u},\bm{v}) \leq C &\left( \sum_{K \in \MTh} \| \curl^{2} \bm{u} 
      \|_{L^2(K)}^2 + \sum_{e \in \MEh} (\| h_e^{-3/2} 
      \jump{\bm{u}} \|^2_{L^2(e)} + \| h_e^{-1/2} 
      \jump{\curl \bm{u} } \|^2_{L^2(e)} \right. \\
      &\left.+ \| h_e^{3/2} \aver{\curl^{3}\bm{u}} \|^2_{L^2(e)} + 
      \| h_e^{1/2} \aver{\curl^{2} \bm{u}}\|^2_{L^2(e)})  \right)^{1/2} \cdot \\
      &\left( \sum_{K \in \MTh} \| \curl^{2} \bm{v} 
      \|_{L^2(K)}^2 + \sum_{e \in \MEh} (\| h_e^{-3/2} 
      \jump{\bm{v}} \|^2_{L^2(e)} + \| h_e^{-1/2} 
      \jump{\curl \bm{v} } \|^2_{L^2(e)} \right. \\
      &\left.+ \| h_e^{3/2} \aver{\curl^{3}\bm{v}} \|^2_{L^2(e)} + 
      \| h_e^{1/2} \aver{\curl^{2} \bm{v}}\|^2_{L^2(e)})  \right)^{1/2} 
    \end{aligned}
  \end{displaymath}
  which completes the proof.
\end{proof}
Now we verify the Galerkin orthogonality.
\begin{lemma}
  Suppose $\bm{u}$ is the exact solution to
  the problem \eqref{eq_quadcurl}, and $\bm{u_h} \in \bmr{U}_h^m$ is the
  numerical solution to the discrete problem \eqref{eq_variational},
  then 
  \begin{equation}
    B_h(\bm{u} - \bm{u_h}, \bm{v_h}) = 0, \quad \forall \bm{v_h} \in \bmr{U}_h^m.
    \label{eq_orthogonality}
  \end{equation}
  \label{le_orthogonality}
\end{lemma}
\begin{proof}
  We first have for the two jump terms
  \begin{displaymath}
    \jump{\bm{u}}|_{e} = \bm{0}, \quad \jump{\curl \bm{u}}|_{e} = 0, \quad
     \quad \forall e \in \MEh^i, \,\, i = 0,1.
  \end{displaymath}
Taking the exact solution into $B_h(\cdot, \cdot)$, we have that
\begin{displaymath}
  \begin{aligned}
    B_h(u, \bm{v_h}) &=\sum_{K \in \MTh} \int_{K}
    \curl^{2} \bm{u}\cdot \curl^{2} \bm{v_h} \d{\bm{x}} +\sum_{K \in \MTh} \int_{K}
    \bm{u}\cdot \bm{v_h} \d{\bm{x}}\\
    &+ \sum_{e \in \MEh} \int_{e}  
    \jump{\bm{v_h}} \cdot \aver{\curl^{3}\bm{u}} \d{\bm{s}}
    +\sum_{e \in \MEh} \int_{e} \jump{\curl \bm{v_h}} 
    \aver{\curl^{2} \bm{u}}  \d{\bm{s}} \\
    &+ \sum_{e \in \MEh^b} \int_{e} \left( g_1 \cdot(\curl^{3}\bm{v_h}
    + \mu_1 \bm{v_h}\times \un) + g_2 \cdot(\curl^{2}\bm{v_h} + \mu_2 \curl\bm{v_h}\times\un) \right)
   \d{\bm{s}}
  \end{aligned}
\end{displaymath}
We multiply the test function $\bm{v_h}$ at both side of equation
\eqref{eq_quadcurl}, and apply the integration by parts to get
\begin{displaymath}
  \begin{aligned}
    \sum_{K \in \MTh} \int_{K} &\bm{f}\cdot \bm{v_h} \d{\bm{x}} = 
    \sum_{K \in \MTh} \int_{K} \curl^{4}\bm{u}\cdot \bm{v_h} \d{\bm{x}} +\sum_{K \in \MTh} \int_{K}
    \bm{u}\cdot \bm{v_h} \d{\bm{x}}\\
    &= \sum_{K \in \MTh} \int_{K}
    \curl^{2} \bm{u}\cdot \curl^{2} \bm{v_h} \d{\bm{x}} +\sum_{K \in \MTh} \int_{K}
    \bm{u}\cdot \bm{v_h} \d{\bm{x}}\\
    &+ \sum_{K \in \MTh} \left( \int_{\partial K} \curl^{3}\bm{u}\cdot
     \bm{v_h}\times\un \d{\bm{s}} + \int_{\partial K}
    \curl^{2} \bm{u}  \cdot (\curl \bm{v_h}\times\un) \d{\bm{s}} \right) \\
    &=\sum_{K \in \MTh} \int_{K}
    \curl^{2} \bm{u}\cdot \curl^{2} \bm{v_h} \d{\bm{x}} +\sum_{K \in \MTh} \int_{K}
     \bm{u}\cdot \bm{v_h} \d{\bm{x}} \\
    &+ \sum_{e \in \MEh} \int_{e}  
    \jump{\bm{v_h}} \cdot \aver{\curl^{3}\bm{u}} \d{\bm{s}}
    +\sum_{e \in \MEh} \int_{e} \jump{\curl \bm{v_h}} 
    \aver{\curl^{2} \bm{u}}  \d{\bm{s}} 
    \d{\bm{s}} 
  \end{aligned}
\end{displaymath}
Thus, by simply calculating, we obtain that
\begin{displaymath}
  B_h(\bm{u_h}, \bm{v_h}) = l_h(\bm{v_h}) = B_h(\bm{u}, \bm{v_h}),
\end{displaymath}
which completes the proof.
\end{proof}
Then we establish the interpolation error estimate of the reconstruction
operator.
\begin{lemma}
  For $0 \leq h \leq h_0$ and $m \geq 2$, there exists a constant $C$ 
  such that
  \begin{equation}
    \DGenorm{\bm{v} - \mc{R} \bm{v}} \leq C \Lambda_m h^{m-1} \| \bm{v}
    \|_{H^{m+1}(\Omega)}, \quad \forall v \in
    H^{s}(\Omega), \quad s = \max(4, m+1).
    \label{eq_interpolation_err}
  \end{equation}
  \label{le_interpolation_err}
\end{lemma}
\begin{proof}
  From Lemma
  \ref{le_localapproximation}, we can show that
  \begin{displaymath}
    \begin{aligned}
      \sum_{K \in \MTh} \| \curl^{2} \bm{v} - \curl^{2}(\mc{R} \bm{v}) \|^2_{L^2(K)} 
      &\leq \sum_{K \in \MTh} C \Lambda_m^2 h_K^{2m-2} \|  \bm{v}
      \|^2_{H^{m+1}(S(K))} \\
      &\leq C \Lambda_m^2 h^{2m-2} \|  \bm{v} \|^2_{H^{m+1}(\Omega)} \leq
      C \Lambda_m^2 h^{2m-2} \| \bm{v} \|^2_{H^{m+1}(\Omega)} \\
    \end{aligned}
  \end{displaymath}
  ,also
  \begin{displaymath}
    \begin{aligned}
      \sum_{K \in \MTh} \|  \bm{v} - (\mc{R} \bm{v}) \|^2_{L^2(K)} 
      &\leq \sum_{K \in \MTh} C \Lambda_m^2 h_K^{2m+2} \|  \bm{v}
      \|^2_{H^{m+1}(S(K))} \\
      &\leq C \Lambda_m^2 h^{2m+2} \|  \bm{v} \|^2_{H^{m+1}(\Omega)} \leq
      C \Lambda_m^2 h^{2m+2} \| \bm{v} \|^2_{H^{m+1}(\Omega)}, \\
    \end{aligned}
  \end{displaymath}
  By the trace estimate \eqref{eq_trace} and the mesh regularity,
  \begin{displaymath}
    \begin{aligned}
      \sum_{e \in \MEh}  h_e^{-1} \| \jump{\curl ( \bm{v} - \mc{R} \bm{v})} 
      \|^2_{L^2(e)} &\leq C \sum_{K \in \MTh} 
      \left( h_K^{-2} \| \curl( \bm{v} - \mc{R}\bm{v}) \|^2_{L^2(K)}
      + \| \curl ( \bm{v} - \mc{R}\bm{v}) \|^2_{H^1(K)} \right) \\
      &\leq C \Lambda_m^2 h^{2m-2} \|  \bm{v}
      \|^2_{H^{m+1}(\Omega)} .
    \end{aligned}
  \end{displaymath}
  and
  \begin{displaymath}
    \begin{aligned}
      \sum_{e \in \MEh}  h_e^{-3} \| \jump{  \bm{v} - \mc{R} \bm{v}} 
      \|^2_{L^2(e)} &\leq C \sum_{K \in \MTh} 
      \left( h_K^{-4} \|  \bm{v} - \mc{R}\bm{v}\|^2_{L^2(K)}
      + h_K^{-2}\|  \bm{v} - \mc{R}\bm{v} \|^2_{H^1(K)} \right) \\
      &\leq C \Lambda_m^2 h^{2m-2} \|  \bm{v}
      \|^2_{H^{m+1}(\Omega)} .
    \end{aligned}
  \end{displaymath}
  which completes the proof.
\end{proof}
Now we are ready to present the \emph{a priori} error estimate under
$\DGenorm{\cdot}$ within the standard Lax-Milgram framework.
\begin{theorem}
  Suppose the problem \eqref{eq_quadcurl} has a
  solution $\bm{u} \in H^{s}(\Omega)$, where $s = \max(4,
  m+1)$, $m \geq 2$ and $\Lambda_m$ has a uniform upper bound 
  independent of $h$.
  Let the bilinear form $B_h(\cdot, \cdot)$ be defined with a
  sufficiently large $\eta$ and $\bm{u_h} \in \bmr{U}_h^m$ be the numerical 
  solution to the problem \eqref{eq_variational}. Then for $h
  \leq h_0$ there exists a constant $C$ such that
  \begin{equation}
    \DGenorm{\bm{u} - \bm{u_h}} \leq C\Lambda_m h^{m-1} \| \bm{u} \|_{H^{m+1}(\Omega)}.
    \label{eq_errorestimate}
  \end{equation}
  \label{th_errorestimate}
\end{theorem}
\begin{proof}
  From \eqref{eq_coercivity}, \eqref{eq_continuity} and
  \eqref{eq_orthogonality}, we have that for any $\bm{v_h} \in \bmr{U}_h^m$,
  \begin{displaymath}
    \begin{aligned}
      \DGenorm{\bm{u_h} - \bm{v_h}}^2 &\leq C B_h(\bm{u_h} - \bm{v_h}, \bm{u_h} - \bm{v_h}) = C
      B_h(\bm{u} - \bm{u_h}, \bm{u_h} - \bm{v_h}) \\
      &\leq C \DGenorm{\bm{u} - \bm{v_h}} \DGenorm{\bm{u_h} - \bm{v_h}}.
    \end{aligned}
  \end{displaymath}
  By the triangle inequality, there holds
  \begin{displaymath}
    \DGenorm{\bm{u} - \bm{u_h}} \leq \DGenorm{\bm{u} - \bm{v_h}} + \DGenorm{\bm{v_h} - \bm{u_h}}
    \leq C \inf_{\bm{v_h} \in \bmr{U}_h^m} \DGenorm{\bm{u} - \bm{v_h}}.
  \end{displaymath}
  Let $\bm{v_h} = \mc{R} \bm{u}$, by the inequality
  \eqref{eq_interpolation_err}, we arrive at 
  \begin{displaymath}
    \DGenorm{\bm{u} - \bm{u_h}} \leq C\Lambda_m \DGenorm{\bm{u} - \mc{R} \bm{u}} \leq C h^{m-1}
    \|\bm{u}\|_{H^{m+1}(\Omega)},
  \end{displaymath}
  which completes the proof.
\end{proof}
\begin{corollary}
  Under the same assumptions, there exists a constant $C$ such that
  \begin{equation}
    \|\bm{u} - \bm{u_h}\|_{L^2(\Omega)} \leq C\Lambda_m h^{m-1} \| \bm{u} \|_{H^{m+1}(\Omega)}.
    \label{eq_errorestimate}
  \end{equation}
  \label{th_errorestimate}
\end{corollary}
\begin{remark}
  We have proved that under the energy norm $\DGenorm{\cdot}$ the
  numerical solution $\bm{u}_h$
  has the optimal convergence rate. It can be seen that
  $\DGenorm{\cdot}$ is stronger than $\|\cdot\|_{L^2(\Omega)}$ 
   from the definition of $\DGenorm{\cdot}$.  Hence, from
  \eqref{eq_errorestimate} we can conclude that the numerical
  solution $\bm{u}_h$ at least have the $O(h^{m-1})$
  convergence rate in $L^2$ norm, which is exactly the convergence rate indicated
  by our numerical tests.
\end{remark}

\section{Numerical Results}
\label{sec_numericalresults}
In this section, we conduct a series numerical experiments to test the
performance of our method. For the accuracy $2 \leq m \leq 4 $, the
threshold $\# S$ we used in the examples
 are listed in Tab.~\ref{tab_patch}. For all examples, 
the boundary $g_1, g_2$ and the right hand side
$f$ in the equation \eqref{eq_quadcurl} are chosen according to the
exact solution. \\

 \textbf{Example 1.} 
We first consider a fourth-order curl problem defined on the squared domain 
$\Omega = (0,1)^2$ . The exact solution is chosen by
\begin{displaymath}
  \bm{u}(x, y) = \begin{bmatrix}
    3\pi \sin^2(\pi y)\cos(\pi y)\sin^3(\pi x)\\
    -3\pi \sin^2(\pi x)\cos(\pi x)\sin^3(\pi y) \\
  \end{bmatrix},
\end{displaymath}

We solve the problem on a sequence of meshes with the size
$h = 1/8, 1/16, 1/32, 1/64, 1/128$. The convergence histories under 
the $\DGnorm{\cdot}$ (which is equivalent to $\DGenorm{\cdot}$ )and
$\| \cdot \|_{L^2(\Omega)}$ are shown in
Fig.~\ref{fig_ex1err}. The error under the energy norm is decreasing 
at the speed $O(h^{m-1})$ for fixed $m$. For $L^2$ error, the speed is
the same as the energy norm. These results
coincide with the theoretical analysis in Theorem
\ref{th_errorestimate} .\\

\begin{table}[htp]
  \begin{center}
\begin{minipage}[t]{0.3\textwidth}
  \centering
  \begin{tabular}{p{0.6cm}|p{0.6cm}|p{0.6cm}|p{0.6cm}}
  \hline\hline
  $m$    & 2 & 3 & 4  \\ \hline
  $\# S$ & 12 & 20 & 27 \\ 
  \hline\hline
  \end{tabular}
\end{minipage}
\hspace{2cm}
\begin{minipage}[t]{0.3\textwidth}
  \centering
  \begin{tabular}{p{0.6cm}|p{0.6cm}|p{0.6cm}}
    \hline\hline
   $m$ & 2 & 3  \\ \hline
   $\# S$ & 25 & 47  \\ 
    \hline\hline
  \end{tabular}
\end{minipage}
\end{center}
\caption{The $\# S$ used in 2D and 3D examples.} 
\label{tab_patch}
\end{table}

\begin{figure}
  \centering
  \includegraphics[width=0.30\textwidth]{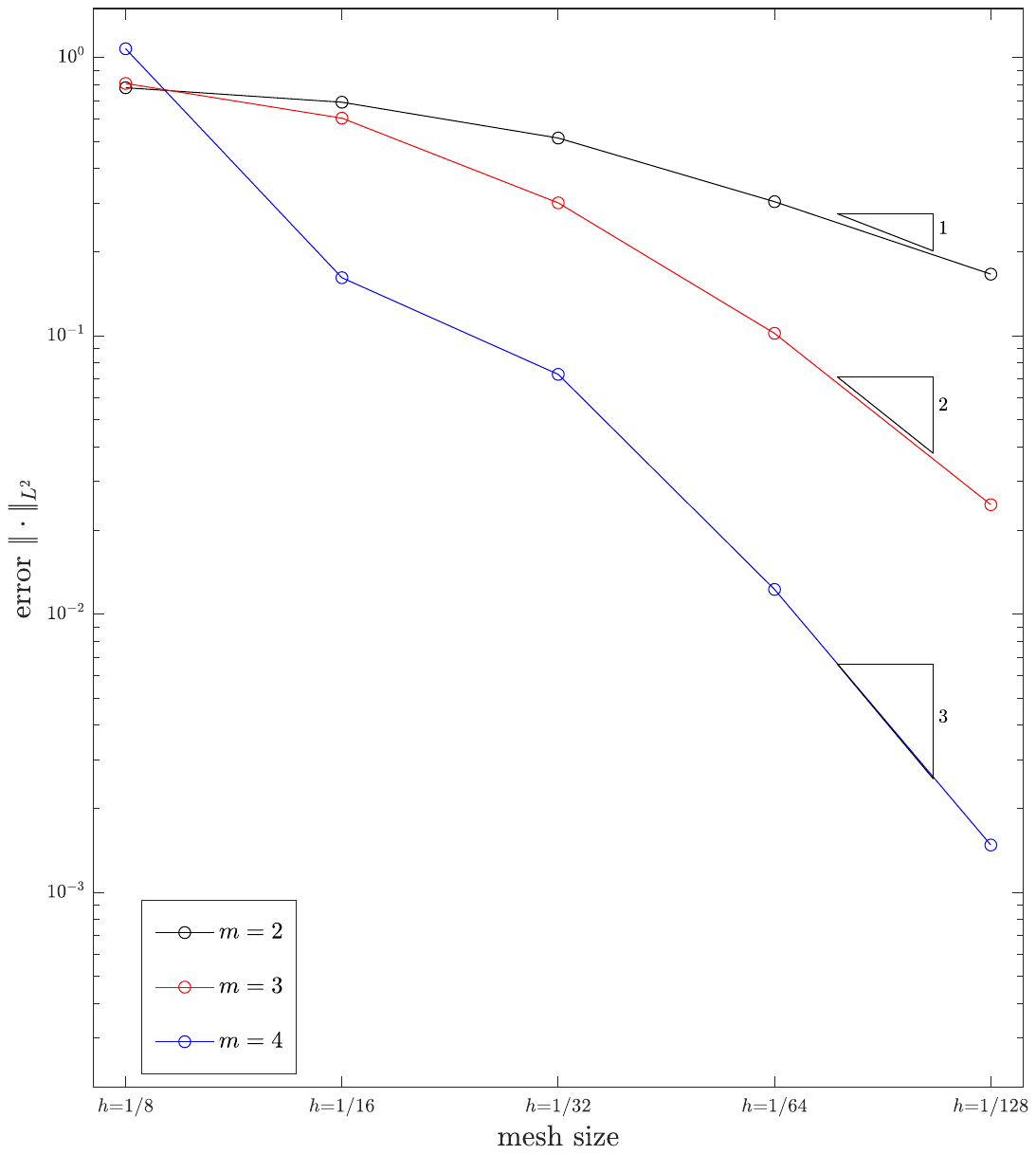}
  \hspace{30pt}
  \includegraphics[width=0.30\textwidth]{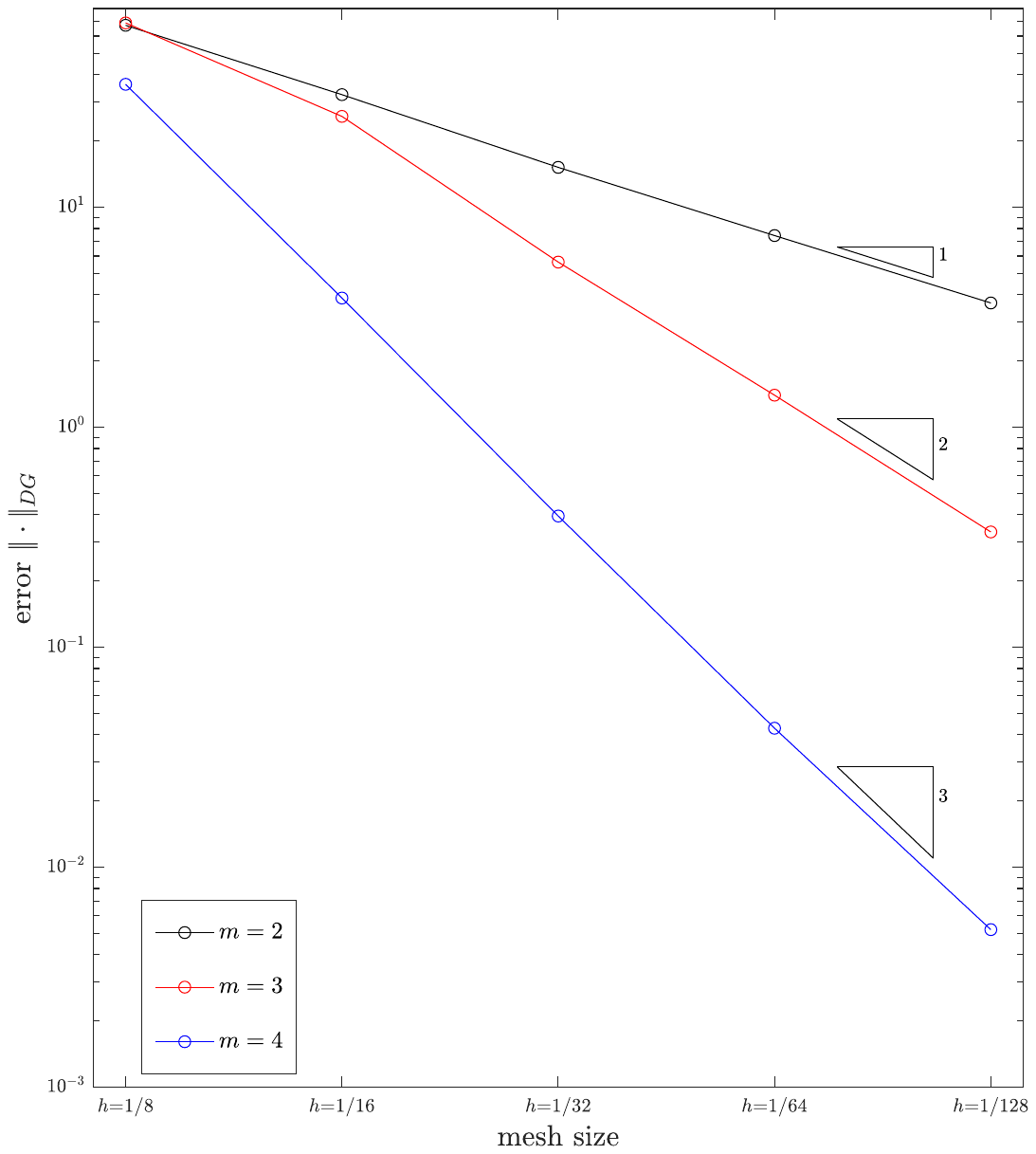}
  \caption{The convergence histories under the $\| \cdot \|_{L^2(\Omega)}$ 
  (left) and the$\DGnorm{\cdot}$ (right) in Example 1.}
  \label{fig_ex1err}
\end{figure}

 \textbf{Example 2.}
We solve a three-dimensional problem in this
case. The computation domain is an unit cube $\Omega = (0,1)^3$
.
We select the exact solution as
\begin{displaymath}
  \bm{u}(x, y, z) = \begin{bmatrix}
    \sin(\pi y)\sin(\pi z)\\
    \sin(\pi z)\sin(\pi x) \\
    \sin(\pi x)\sin(\pi y)\\    
  \end{bmatrix},
\end{displaymath}
We use the tetrahedra meshes generated by the Gmsh software
\cite{geuzaine2009gmsh}. We solve the problem on three different meshes
with the reconstruction order $m = 2,3$. The convergence order
in both norms is shown in
Fig.~\ref{fig_ex5err}, also consistent with our
theoretical predictions.

\begin{figure}[htbp]
  \centering
  \includegraphics[width=0.30\textwidth]{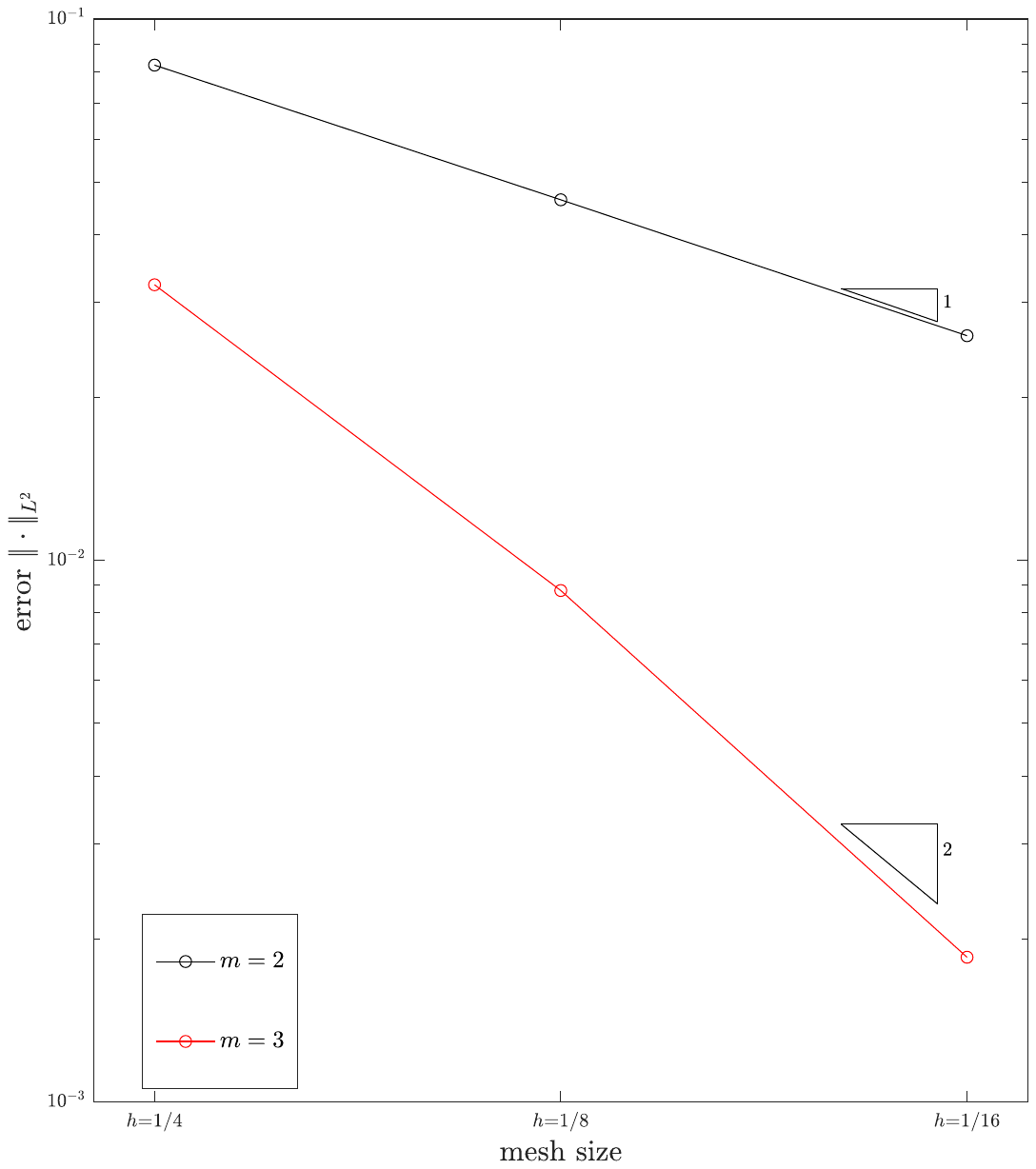}
  \hspace{30pt}
  \includegraphics[width=0.30\textwidth]{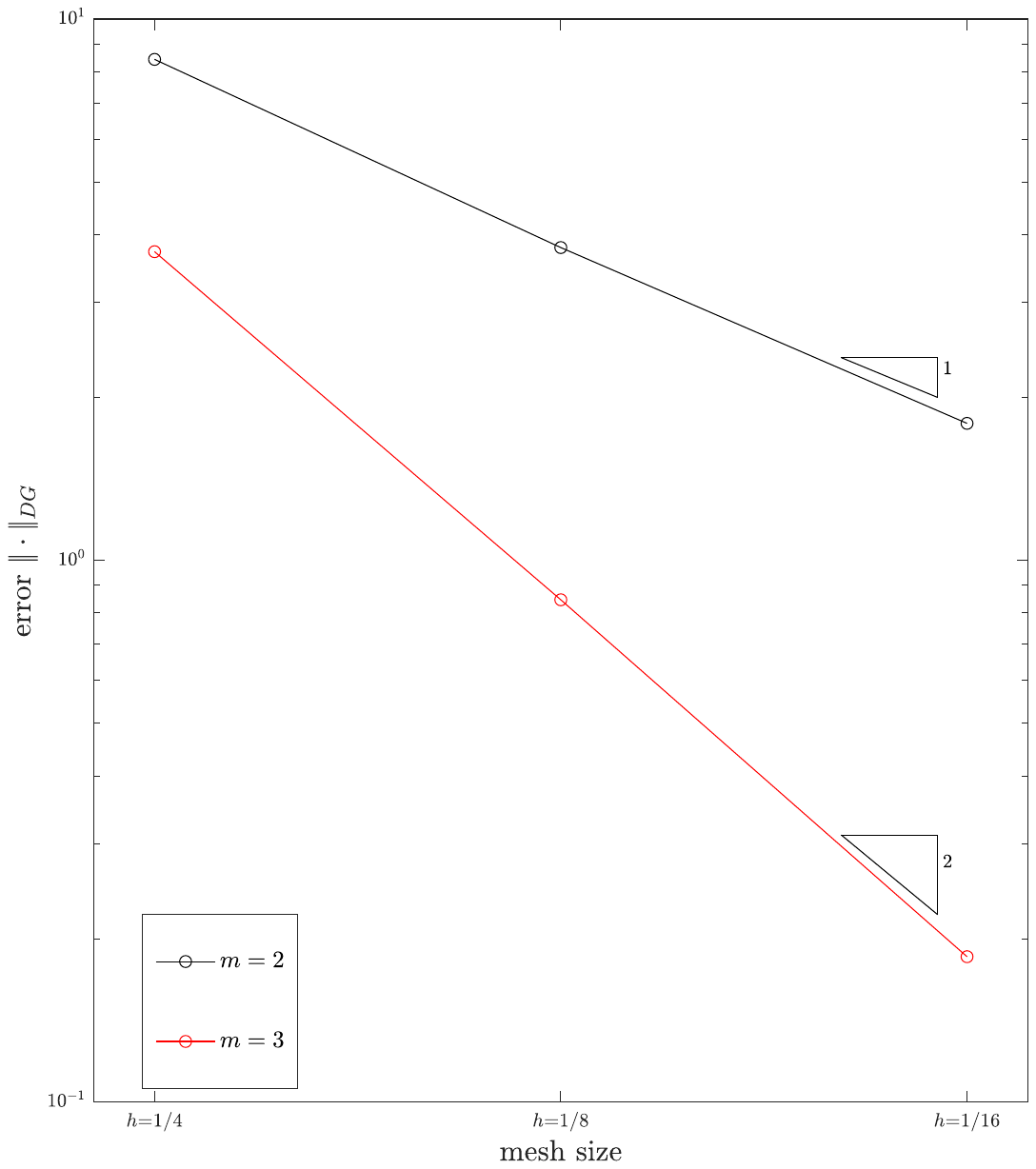}
  \caption{The convergence histories under the $\| \cdot \|_{L^2(\Omega)}$
  (left) and the  $\DGnorm{\cdot}$(right) in Example 2.}
  \label{fig_ex5err}
\end{figure}

\section{Conclusion}
\label{sec_conclusion}
In this paper, we proposed an IPDG method
for the fourth-order curl problem with the reconstructed discontinuous
approximation. The approximation space is based on
patch reconstruction from piecewise constant sapce and can approximate 
functions up to high order accuracy. We show numercial experiments
in two and three dimensions to examine the order of convergence under the energy
norm and the $L^2$ norm.


\begin{appendix}
  \section{Calculating the Reconstruction Constant}
  \label{sec_reconM}
  This appendix gives the method to compute the constant $\Lambda_m$
  for a given mesh $\MTh$. For any element $K \in \MTh$, let $p_1,
  p_2, \ldots, p_l$ be a group of standard orthogonal basis functions
  in $\mb{P}_m(K)$ under the $L^2$ inner product $(\cdot,
  \cdot)_{L^2(K)}$. Then, any polynomial $q \in \mb{P}_m(K)$ can be
  expanded by a group coefficients $\bm{a} = \{a_j\}_{j = 1}^l \in
  \mb{R}^{l}$ such that $q =
  \sum_{j = 1}^l a_j p_j$. We can naturally extend $q$ and all $p_j$
  to the domain $S(K)$ by the polynomial extension. The main step to
  get $\Lambda_m$ is to compute $\Lambda_{m, K}$ for all elements, and
  $\Lambda_{m, K}$ can be represented as 
  \begin{displaymath}
    \Lambda_{m, K}^2 = \max_{\bm{a} \in \mb{R}^l} \frac{|
    \bm{a}|_{l^2}^2 }{h_K^d \bm{a}^T B_K \bm{a}}, \quad B_K =
    \{b_{ij}\}_{l \times l}, \quad b_{ij} = \sum_{\bm{x} \in I(K)}
    p_i(\bm{x})p_j(\bm{x}).
  \end{displaymath}
  From the matrix representation, $\Lambda_{m, K} = (h_K^d
  \sigma_{\min}(B_K))^{-1/2}$, where $\sigma_{\min}(B_K)$ is the
  smallest singular value to $B_K$. 
  Hence, it is enough to observe the
  minimum value of $\sigma_{\min}(B_K)$ for all elements, and
  $\Lambda_m$ can be computed by \eqref{eq_Lambdam}.

  As we remarked in Section \ref{sec_space}, when the patch is wide
  enough, $\Lambda_m$ will admit a uniform upper bound independent of
  the mesh size. Here, we will show $\Lambda_m$ for different size of
  the patch. We consider the triangular mesh with $h = 1/40$
  and the tetrahedral mesh with $h = 1/16$ in two and three
  dimensions, which are used in Example 1 and Example 2. 
  The values of $\Lambda_m$ are collected in Fig.~\ref{fig_const}. 

  \begin{figure}
    \centering
    \includegraphics[width=0.25\textwidth]{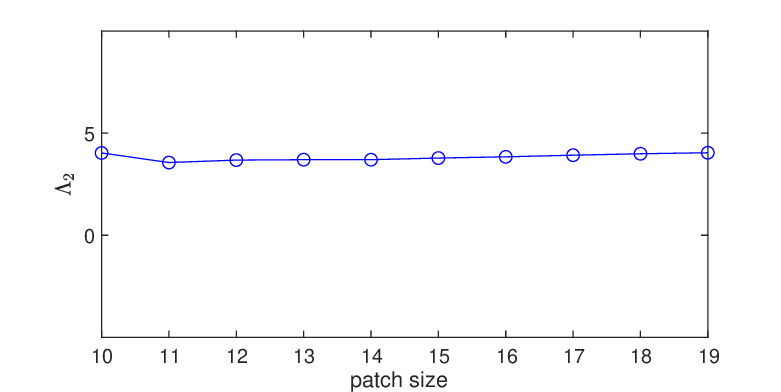}
    \hspace{20pt}
    \includegraphics[width=0.25\textwidth]{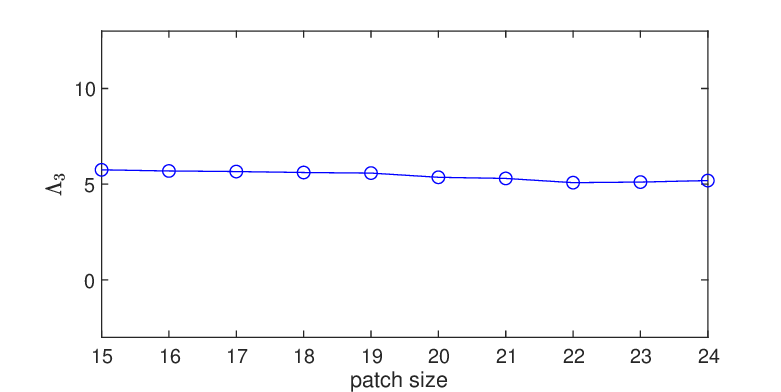}
    \hspace{20pt}
    \includegraphics[width=0.25\textwidth]{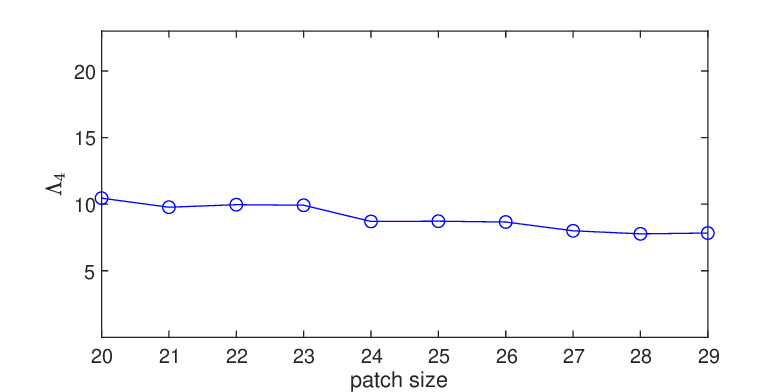}
    \hspace{20pt}
    \includegraphics[width=0.25\textwidth]{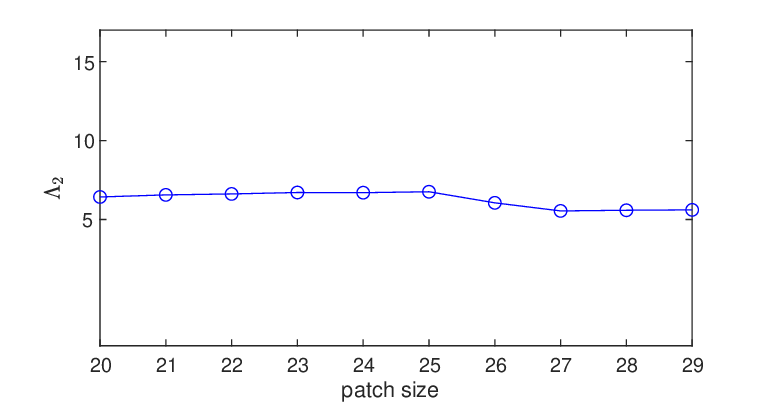}
    \hspace{20pt}
    \includegraphics[width=0.25\textwidth]{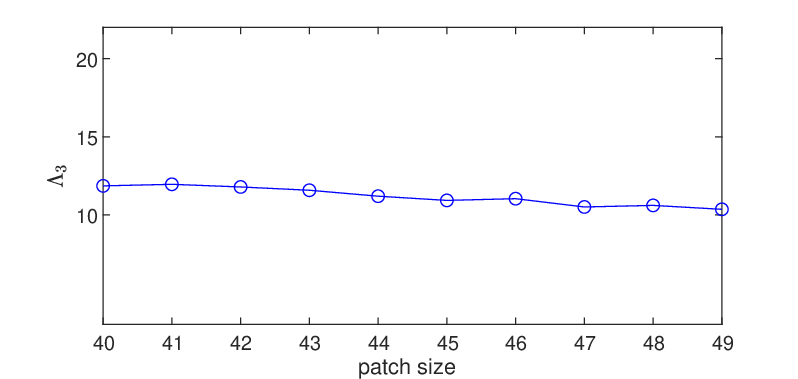}

    \caption{ $\Lambda_m$ in 2d with $ h = 1/40$ (row 1) / 
    $\Lambda_m$ in 3d with $h = 1/16$ (row 2).}
    \label{fig_const}
  \end{figure}

\end{appendix}

\bibliographystyle{amsplain}
\bibliography{../ref}

\end{document}